\documentclass[11pt]{amsart}

\usepackage[colorlinks=true, linkcolor=blue, anchorcolor=blue, citecolor=blue, filecolor=blue, menucolor= blue, urlcolor=red]{hyperref}

\usepackage{amsmath,amssymb,amsthm}
\usepackage{verbatim}
\usepackage{comment}
\usepackage{graphicx}
\usepackage{pgf,tikz}
\usetikzlibrary{arrows}

\long\def\eatit#1{}
\newtheorem{thm}{Theorem}[section]

\newtheorem{lem}[thm]{Lemma}

\newtheorem{conj}[thm]{Conjecture}

\theoremstyle{definition}

\newcommand{\drawingOne}{1}
\newcommand{\pr}[1]{{{\bf P}^{#1}}}

\newcommand{\field}{\mathbb K}

\newcommand{\reg}{\operatorname{reg}}
\newcommand{\ha}[1]{\widehat{\alpha}(#1)}
\newcommand{\depth}{\mathop{\mathrm{depth}}\nolimits}
\newcommand{\pd}{\mathop{\mathrm{pd}}\nolimits}

\newcommand\hatrho{\widehat{\rho}}

\begin{document}
\title[Resurgences for ideals of special point configurations in $\pr N$]{Resurgences
for ideals of special point configurations in $\pr N$
coming from hyperplane arrangements}

\author{M.\ Dumnicki}
%\author{Marcin Dumnicki}
\address{Jagiellonian University\\
Institute of Mathematics\\
{\L}ojasiewicza 6\\
30-384 Krak\'ow}
\email{Marcin.Dumnicki@im.uj.edu.pl}
\author{B.\ Harbourne}
%\author{Brian Harbourne}
\address{Department of Mathematics\\
University of Nebraska\\
Lincoln, NE 68588-0130 USA}
\email{bharbour@math.unl.edu}
\author{U.\ Nagel}
%\author{Uwe Nagel}
\address{Department of Mathematics\\
University of Kentucky\\
715 Patterson Office Tower\\
Lexington, KY 40506-0027 USA}
\email{uwe.nagel@uky.edu}
\author{A.\ Seceleanu}
%\author{Alexandra Seceleanu}
\address{Department of Mathematics\\
University of Nebraska\\
Lincoln, NE 68588-0130 USA}
\email{aseceleanu@unl.edu}
\author{T.\ Szemberg}
%\author{Tomasz Szemberg}
\address{Instytut Matematyki UP\\
Podchor\c a\.zych 2\\
PL-30-084 Krak\'ow, Poland}
\email{tomasz.szemberg@uni-due.de}
\author{H.\ Tutaj-Gasi\'nska}
%\author{Halszka Tutaj-Gasi\'nska}
\address{Jagiellonian University\\
Institute of Mathematics\\
{\L}ojasiewicza 6\\
PL-30-348 Krak\'ow}
\email{htutaj@im.uj.edu.pl}

\thanks{The second author's work on this project
was sponsored by the National Security Agency under Grant/Cooperative
agreement ``Topics in Algebra and Geometry, Curves and Containments'' Number H98230-13-1-0213.
The United States Government is authorized to reproduce and distribute reprints
notwithstanding any copyright notice. The work of the third author was partially  
supported by the National Security Agency under Grant 
Number H98230-12-1-0247. The fifth author's research was partially supported
by NCN grant UMO-2011/01/B/ST1/04875.}

\begin{abstract}
Symbolic powers of ideals have
attracted interest in commutative algebra and algebraic geometry for
many years, with a notable recent focus on containment relations between
symbolic powers and ordinary powers; see for
example \cite{refBH, refCu, refELS, refHaHu, refHoHu, refHu1, refHu2}
to cite just a few. Several invariants have been introduced and studied in the latter context,
including the resurgence and asymptotic resurgence
\cite{refBH,refGHvT}.

There have been exciting new developments in this area recently.
It had been expected for several years that $I^{Nr-N+1}\subseteq I^r$
should hold for the ideal $I$ of any finite set of points in $\pr N$ for all $r>0$,
but in the last year various counterexamples have now been constructed
(see \cite{refDST,refHS,RefC. et al}), all involving point sets
coming from hyperplane arrangements.
In the present work, we compute their resurgences and obtain in particular the
first examples where the resurgence and the asymptotic resurgence are not equal.
\end{abstract}

\date{April 19, 2014}

 %Projective techniques


\maketitle

\section{Introduction}
In commutative algebra, ideals are major objects of interest, often given directly
by specifying generators. Ideals are also important objects
of study in algebraic geometry, but the ideals are specified indirectly, often in terms of
vanishing conditions. Thus in commutative algebra it is quite natural to study
the behavior of powers of ideals, but in algebraic geometry it is more natural to
study symbolic powers. For example, given a finite set $S\subset {\bf P}^N$ of points
in projective space (over a field $\field$), we have the polynomial ring
$R=\field[{\bf P}^N]$
in $N+1$ variables over $\field$. The ideal $I_S\subseteq \field[{\bf P}^N]=R$
is the ideal generated by all homogeneous polynomials (i.e., forms)
vanishing on $S$. If $\mathcal{I}_S$ is the ideal sheaf on ${\bf P}^N$
corresponding to $I_S$, then the $m$th symbolic power $I_S^{(m)}$ is
canonically isomorphic to $\bigoplus_t H^0({\bf P}^N,\mathcal{I_S}^m(t))$.
Alternatively, $I_S^{(m)}$ is generated by
all forms vanishing to order at least $m$ at each point of $S$; i.e.,
if $S=\{p_1,\ldots,p_s\}$, then $I_S=\cap_iI_{p_i}$ and $I_S^{(m)}=\cap_iI_{p_i}^m$.
The precise relationship between $I_S^m$ and $I_S^{(m)}$ is that
$I_S^m=I_S^{(m)}\cap Q$ where $Q$ is primary for
the irrelevant ideal $M\subset \field[{\bf P}^N]$ (i.e., the maximal
homogeneous ideal, this being the one generated by the variables of the polynomial ring
$\field[{\bf P}^N]$). Algebraically, taking powers of an ideal can introduce
adventitious primary components; recovering the symbolic power from the ordinary power
requires removing these adventitious components. This leads to the general definition
of symbolic power, namely the $m$th symbolic power $I^{(m)}$ of an ideal $I\subseteq R$
is defined to be
$I^{(m)}=R\cap(\cap_{P\in\operatorname{Ass}(I)}I^mR_P)$
(where the intersection takes place in $R_{(0)}$).

It is immediately apparent from this discussion that one always has
$I_S^m\subseteq I_S^{(m)}$. There are sets of points $S$ for which
all powers of $I_S$ are symbolic (i.e., such that
$I_S^m=I_S^{(m)}$ holds for all $m>0$), but it is an open problem to
characterize those $S$ with this property, and there are also easy examples
of $S$ where equality sometimes fails, so nontrivial $M$-primary components
$Q$ really do occur.

When $I_S^r\subsetneq I_S^{(r)}$, it is at least true
for $m$ sufficiently large (such as for $m$ greater than or
equal to the maximum of $r$ and the saturation degree of $I_S^r$) that we have
$I_S^{(m)}\subseteq I_S^r$, but it is much less obvious what the least such $m$ is.
A quantity known as the \emph{resurgence} was introduced in \cite{refBH}
to study this issue.
Let $(0)\neq I\subsetneq R=\field[{\bf P}^N]$ be a homogeneous ideal.
Then the resurgence $\rho(I)$ of $I$ is defined to be
$$\rho(I)=\sup\Big\{\frac{m}{r}:I^{(m)}\not\subseteq I^r\Big\}.$$
Its asymptotic version $\hatrho(I)$ is defined as
$$\hatrho(I)=\sup\Big\{\frac{m}{r}:I^{(mt)}\not\subseteq I^{rt}\;\mbox{ for } t\gg 0\Big\}.$$
It is immediate that
$$\hatrho(I)\leq \rho(I).$$
Whereas it might be expected that these two invariants differ, no examples
of ideals where this actually happens have been known up to now.
In this note we compute examples showing that a strict inequality between these two invariants
can occur.

A priori it seems possible that $\rho(I)$ could be infinite.
However, given $r\geq1$, a fundamental result of \cite{refHoHu, refELS}, is that
\begin{equation}\label{ELSHH}
I^{(m)}\subseteq I^r\text{ for $m\geq Nr$ for all homogeneous ideals $I\subseteq \field[{\bf P}^N]$. }
\end{equation}
This shows that $\rho(I)\leq N$ for nontrivial ideals $I$.
   On the other hand for a nontrivial ideal $I$ we have always $\hatrho(I)\geq 1$
   by \cite[Theorem 1.2]{refGHvT}.

No examples are known for which $\rho(I)=N$, but examples from
\cite{refBH} show that ideals $I$ can be given with $\rho(I)$ arbitrarily close to $N$.
Thus no expression of the form $m>cr$ for constant $c<N$ can ensure containment
$I^{(m)}\subseteq I^r$ for all homogeneous ideals $I\subseteq R$ and all $r$.
This still leaves open the question of
whether there are lower bounds on $m$ smaller than $Nr$ guaranteeing
containment $I^{(m)}\subseteq I^r$ for all $I$ and $r$.

For example, if $I$ is an ideal of points in ${\bf P}^2$, then we have
$I^{(2r)}\subseteq I^r$ and hence $I^{(4)}\subseteq I^2$.
C.\ Huneke asked if $I^{(3)}\subseteq I^2$ also always holds for
ideals $I$ of finite sets of points in the plane.
This led to the following (now known to be false)
conjecture of the second author \cite{refB. et al}
as a possible improvement on (\ref{ELSHH}):

\begin{conj}\label{Bconj}
The containment $I^{(rN-(N-1))}\subseteq I^r$
holds for all homogeneous ideals in $\field[{\bf P}^N]$.
\end{conj} 

The containment of Conjecture \ref{Bconj}
does indeed hold for many ideals $I$ for many $r$ and $N$
(see for example, \cite{refBCH, refB. et al, refHaHu}), including 
for ideals of finite sets of general points when $N=2,3$ \cite{refBH,refD},
but it is now known that failures can occur. The first failure found is that of
\cite{refDST} showing that $I^{(3)}\not\subseteq I^2$ occurs for the ideal of a certain
configuration of twelve points in ${\bf P}^2$ over the field $\field={\bf C}$ of complex numbers.
These twelve points are dual to the twelve lines meeting a smooth plane cubic curve
only at the flex points of the cubic, and thus have the combinatorially
interesting property of there being nine lines passing through
subsets of exactly four of the twelve points, and for each of the twelve
points there is a subset of exactly three of the nine lines which vanish at the point.
Any twelve of the 13 points of ${\bf P}^2$ over the finite field $\field$ of three elements
also have this same combinatorial structure, and the ideal $J$ of these points
also has $J^{(3)}\not\subseteq J^2$ (see \cite{refBCH};
for additional counterexamples to
Conjecture \ref{Bconj}, for various values of $N$ and $r$, see \cite{refHS}).
However, the resurgences distinguish the two ideals; indeed,
$\rho(I)=3/2$ and $\hatrho(I)=4/3$, while
$\rho(J)=\hatrho(J)=5/3$ (see Theorem \ref{thm:fermat configurations}
and \ref{mainThmSec3}).

Recently a new counterexample with $N=r=2$
has been announced \cite{RefC. et al}, which can be constructed over the rationals (see
Figure \drawingOne).
Its combinatorial structure is different from those of \cite{refDST, refBCH} mentioned above,
and the asymptotic resurgence is different for all three,
but interestingly, its resurgence turns out to be the same as that of
\cite{refDST} (see
Theorem \ref{thm:fermat configurations} and Theorem \ref{realExample}).
The asymptotic resurgence, surprisingly, thus is
perhaps a more sensitive invariant for
differentiating between various counterexamples.

The goal of this note is to compute $\rho(I)$ and $\hatrho(I)$ for various ideals $I$
giving counterexamples
to Conjecture \ref{Bconj} for ideals of points in ${\bf P}^N$, including those of
\cite{refDST, refBCH,RefC. et al} and some of those of \cite{refHS}.

\section{Results specific to the plane}

Up to choice of coordinate variables $x,y$ and $z$ on ${\bf P}^2$, the ideal $I$
of \cite{refDST} for which $I^{(3)}\not\subseteq I^2$
can be taken to be $I=(x(y^3 - z^3), y(z^3 - x^3), z(x^3 - y^3))$.
More generally, for $n\geq3$
and $\field$ any field of characteristic not equal to 2 but containing $n$
distinct roots of 1, then $I^{(3)}\not\subseteq I^2$ holds for the ideal
$I=(x(y^n - z^n), y(z^n - x^n), z(x^n - y^n))\subset \field[x,y,z]$ (see \cite{refHS});
we note that $I$ is the ideal of a certain very special set of $n^2+3$ points of ${\bf P}^2$,
these being the three coordinate vertices in addition to a complete intersection of $n^2$ points.
We begin by computing the resurgence of these ideals.

To this end it is useful to recall Waldschmidt's constant. For a homogeneous ideal
$(0)\neq J\subsetneq R=\field[{\bf P}^N]$, Waldschmidt's constant
$\ha{J}$ is defined to be the following limit:
\begin{equation}\label{eq:waldschmidt}
\ha{J}=\lim_{m\to\infty}\frac{\alpha(J^{(m)})}{m}=\inf_{m\geq 1}\frac{\alpha(J^{(m)})}{m},
\end{equation}
where $\alpha(J^{(m)})$ is the least degree of a nonzero homogeneous
element of $J^{(m)}$. (The existence of the limit and
the equality to the infimum follows from sub-additivity of $\alpha$;
see \cite[Lemma 2.3.1]{refBH}.)

The connection between the various invariants has been discussed in \cite[Theorem 1.2]{refGHvT}. In particular we have
\begin{equation}\label{eq:lower bound on hatrho}
\frac{\alpha(I)}{\ha{I}}\leq \hatrho(I)\leq \rho(I).
\end{equation}

We are now in position to prove our first main result.

\begin{thm}\label{thm:fermat configurations}
   Let $I=(x(y^n - z^n), y(z^n - x^n), z(x^n - y^n))\subset R=\field[x,y,z]$
where $n\geq3$
and $\field$ is any field of characteristic not equal to 2 containing $n$
distinct roots of 1. Then
   $$\hatrho(I)=\frac{n+1}{n} \;\mbox{ and }\; \rho(I) = 3/2.$$
\end{thm}

\begin{proof}
Since $I^{(3)}\not\subseteq I^2$ by \cite{refHS}, we have $3/2\leq \rho(I)$.
We will show that also $\rho(I)\leq 3/2$ and hence $\rho(I)=3/2$.
From \cite[Lemma 2.3.4]{refBH} we know that
   $\alpha(I^{(m)}) > \reg(I^r)$ implies $I^{(m)}\subseteq I^r$. But $\alpha(I^{(m)}) \geq m\ha{I}$ by \eqref{eq:waldschmidt}
and we will show momentarily that $\ha{I}=n$. By
\cite[Theorem 1.7.1]{refC1} or \cite[Theorem 0.5]{refC2}, we have
\begin{equation}\label{MC}
\reg(I^r)\leq 2\reg(I)+(r-2)\omega(I),
\end{equation}
where $\omega(I)$ is the maximum among the degrees of
a minimal set of homogeneous generators of $I$.
In separate work by Nagel and Seceleanu still in preparation,
the minimal free resolutions of $I^r$ have been
determined for all $r\geq1$. The following resolution is the special case of this result
obtained using their argument with $r=1$.
By the Hilbert-Burch Theorem, the minimal free graded resolution of $I$ is
$$0\to R(-2n)\oplus R(-n-3) \to R(-n-1)^3 \to I \to 0.$$
Indeed, set $A=\left[\begin{matrix} xy & xz & yz \\ z^{n-1} & y^{n-1} & x^{n-1}\end{matrix}\right]^T$
and note that the ideal $I$ is generated by
the maximal minors of $A$. Furthermore, since $I$ is the defining ideal of
a reduced set of points in $\pr{2}(\field)$, we have $\dim(R/I)=\depth(R/I)=1$
and so the projective dimension has $\pd(R/I)=2$. Now the Hilbert-Burch theorem
guarantees that
$0\to R(-2n)\oplus R(-n-3) \stackrel{A}{\longrightarrow} R(-n-1)^3 \to R \to R/I \to 0$
is a minimal free resolution of $R/I$, which implies that the minimal resolution of $I$ fits the description above.

Thus $\reg(I) = 2n-1$ and $\omega(I)=n+1=\alpha(I)$, so the bound in (\ref{MC}) becomes
   $$\reg(I^r)\leq 4n-2+(r-2)(n+1)=r(n+1)+2(n-2).$$

   \textbf{Claim}:  $\ha{I}=n$. \\
   To see this, note that
   $I$ is contained in the complete intersection ideal
   $J=(y^n - z^n, z^n - x^n)$  of $n^2$ points.
   Thus
   $$\alpha(I^{(3m)}) \geq \alpha(J^{(3m)}) = \alpha(J^{3m}) = 3m\alpha(J) = 3mn.$$
   But $((x^n-y^n)(x^n-z^n)(y^n-z^n))^m$ is in $I^{(3m)}$ so $3mn \geq \alpha(I^{(3m)})$.
   Thus $\alpha(I^{(3m)}) = 3mn$, hence $\ha{I} = n$.

Now, $r\geq 4$ is equivalent to $3rn/2\geq (n+1)r + 2(n-2)$, so
for $m/r>3/2$ and $r\geq 4$ we obtain
  $$\alpha(I^{(m)})\geq m\ha{I} = mn > 3rn/2\geq (n+1)r + 2(n-2) \geq \reg(I^r)$$
and hence we have $I^r\subseteq I^{(m)}$ whenever $r\geq4$ and $m/r>3/2$.
If $r=2$ but $m/r> 3/2$, then $m\geq4$, hence in this case we have
$I^{(m)}\subseteq I^r$ by (\ref{ELSHH}).

We are left with the case of $r=3$ and so $m\geq5$;
if $I^{(5)}\subseteq I^3$ (and hence $I^{(m)}\subseteq I^{(5)}\subseteq I^3$),
then $I^{(m)}\subseteq I^r$ for all $m$ and $r$ with $m/r>3/2$,
hence $\rho(I)\leq3/2$ and so $\rho(I) = 3/2$.
Thus we now check that $I^3$ contains $I^{(5)}$.
We have
   $$\alpha(I^{(5)})\geq 5\ha{I} = 5n >5n-1 = 3(n+1) + 2(n-2) \geq \reg(I^3),$$
so $I^3$ indeed contains $I^{(5)}$.

   The asymptotic resurgence of $I$ is easily established taking into account
   that the upper bound
   \begin{equation}\label{eq:upper bound on hatrho}
      \hatrho(I)\leq\frac{\omega(I)}{\ha{I}}
   \end{equation}
   (which was established in \cite[Theorem 1.2]{refGHvT})
   agrees in our situation with the lower bound stated in \eqref{eq:lower bound on hatrho}.
\end{proof}

Next we consider the example constructed in \cite{RefC. et al}.
Figure \drawingOne\ shows the example. It consists of 12 lines with 19 triple points (and 9
double points, which we ignore). The configuration as considered in \cite{RefC. et al}
used a specific set of points defined over the reals, but in fact the points can be defined
over the rationals (or any field $\field$ large enough to accommodate the desired
combinatorial structure of the arrangement of lines).
This is because one has some freedom in choosing the points.
This is indicated in Figure \drawingOne\ by representing
the points $A$, $B$ and $C$ as open circles; these points
are free to be placed anywhere, as long as they are not collinear. The three points shown
as triangles ($D$, $E$ and $F$) are required to lie on the lines
through pairs of the points $A$, $B$ and $C$ but are otherwise (mostly) free.
The other points are determined in terms of these 6.
By fixing an appropriate choice of coordinates, we see there is in fact
a single degree of freedom, represented in our construction below
by the parameter $t$. The specific example considered in \cite{RefC. et al}
is the one for which all of the points are affine and the points
$E,F$ and $L$ in Figure \drawingOne\ form an equilateral triangle.
It corresponds (up to a choice of coordinates) to choosing
our parameter $t$ to be $t=-\frac{\sqrt{3}-1}{2}$
(as is easy to see by computing cross ratios for the points
$F$, $B$, $K$ and $C$). Note however, that
for some values of $t$, the configuration of points becomes degenerate
(for example, some of the points can coincide, as we will see below),
and so some values of $t$ are not allowed.

So here is the construction: take three general points
$A, B, C\in\pr2$, as shown in Figure \drawingOne. We may assume
that $A=[0:0:1]$, $B=[0:1:0]$ and $C=[1:0:0]$. We may also assume $\field[\pr2]=\field[x,y,z]$,
where $x=0$ is the line $AB$ through $A$ and $B$, $y=0$ is the line $AC$
through $A$ and $C$, and $z=0$
is the line $BC$ through $B$ and $C$. Now pick general points $D\in AB$, $E\in AC$
and $F\in BC$. By appropriate choice of coordinates, we may assume $D=[0:1:1]$ and
$E=[1:0:1]$, but this fixes the coordinate system on $\pr2$, so now $F$ must be written as
$F=[1:t:0]$, for some parameter $t$, which can either be in $\field$ or in some extension field of
$\field$. (However, not all values of $t$ are allowed. If $t=0$, then $F=C$, but as Figure
\drawingOne\ shows, $F$ and $C$ should be distinct. As we will see below, we also
need $t\neq -1,-2$: if $t=-1$, then $F=K$ and $DE=NO$, while if $t=-2$, then $S=D$.
Also, we must avoid $t^2+t+1=0$, since in that case $M=N=C$.)

With these choices, $BE$ is $x-z=0$, $AF$ is $tx-y=0$,
$DF$ is $tx-y+z$ and $DE$ is $x+y-z=0$. Then we obtain the following points, shown
in  Figure \drawingOne:
$G=[1:t:1]$ is the point $AF\cap BE$,
$H=[1:t+1:1]$ is the point $DF\cap BE$,
$I=[1:0:-t]$ is the point $DF\cap AC$,
$J=[1:t:t+1]$ is the point $AF\cap DE$, and
$K=[1:-1:0]$ is the point $BC\cap DE$.
Then $HJ$ is the line $(t^2+t+1)x-ty-z=0$ and
$L$ is the point $[0:1:-t]= HJ\cap AB$,
$M$ is the point $[1:0:t^2+t+1]= HJ\cap AC$, and
$N$ is the point $[t:t^2+t+1:0]=HJ\cap BC$.
Next, $IK$ is the line $tx+ty+z=0$,
$O$ is the point $[t:-(t+1):t]= IK\cap BE$ and
$P$ is the point $[1:t:-(t^2+t)]= IK\cap AF$.
(Note that $L$ has already been defined as the point $HJ\cap AB$,
but it is easy to check that $L$ is also on $IK$ and is thus
the point of intersection of all three lines, $HJ, AB, IK$, as shown in Figure \drawingOne.)
We now get the line $GM$: $(t^3+t^2+t)x-(t^2+t)y-tz=0$ and the points
$Q=[0:-t:t^2+t]=[0:-1:t+1]=GM\cap AB$ and
$R=[t^2+2t:t^3+2t^2+t:t]= [t+2:t^2+2t+1:1]= GM\cap DF$, followed by
the line $NO$: $(t^2+t+1)x-ty-(t^2+2t+2)z=0$ (note that $R$ is on $NO$,
hence $R$ is the the point of intersection of $GM, DF, NO$). The 19th and final point
is $S=[t^2+3t+2:-(t+1):t^2+2t+1] = [t+2 :-1: t+1] =  DE\cap NO$.
(Note that if $t=-1$, then $DE=NO$, so $S$ is not defined, and if $t=-2$, then
$S=[0:1:1]=D$.)
There is one last line, $CQ$: $(t+1)y+z=0$, and it is easy to check that
$P$ and $S$ are on $CQ$.

(As an aside we also mention that there are 10 conics through
sets of 6 points, as can be seen directly if one carries out the construction above
using, for example, the software Geogebra, available on-line for free
and which we used to create Figure \drawingOne. Each of the points $A,H,K,B,D,E,F,I,J,L$
is a triple point, but the union of the three lines through any one of these
10 points contains only 13 of the 19 points $A,\ldots,S$. The missing 6
lie on a conic, reducible for the points $A$, $H$ and $K$.)

Given any field $\mathbb F$, one can construct the ideal $I\subset \mathbb F(t)[x,y,z]$
of the points $A,\ldots,S$ using software such as
Singular \cite{Sing} (see the script provided in \cite{RefC. et al})
or Macaulay 2 \cite{refM2} (code is included as commented out text
in the \TeX\ source file for this article). When $\mathbb F={\bf Q}$ is the rationals,
so $\field=\mathbb F(t)$ for an indeterminate $t$, one finds that
$I$ is generated by 3 quintics
and has $\alpha(I)=\omega(I)=5$ and $\reg(I)=7$, and that $I^{(3)}\not\subseteq I^2$
(this failure of containment can be checked fairly efficiently by checking that
the product of the forms defining the 19 lines, which clearly is in $I^{(3)}$,
is not in $I^2$). In fact, the same results will hold for $\field={\bf Q}$
by taking $t$ to be any sufficiently general element of either ${\bf Q}$ or
of an extension field of $\bf Q$.
One can even take $\field$ to be a finite field.
For example, for $\field={\bf Z}/{31991\bf Z}$ and $t=5637$ (a specific but randomly
chosen value), Macaulay 2 shows that the points $A,\ldots,S$
are distinct and that $I$ again satisfies $\alpha(I)=\omega(I)=5$, $\reg(I)=7$ with $I^{(3)}\not\subseteq I^2$.

\begin{figure}[h]
\begin{tikzpicture}[line cap=round,line join=round,>=triangle 45,x=.75cm,y=.75cm]
%\begin{tikzpicture}[line cap=round,line join=round,>=triangle 45,x=1.0cm,y=1.0cm]
\clip(-6.295541561786057,-2.7118849043725834) rectangle (7.507390761802504,8.288421566710701);
\draw [domain=-6.295541561786057:7.507390761802504] plot(\x,{(--4.308854742566771-3.7781170691944217*\x)/1.4573087803652975});
\draw [domain=-6.295541561786057:7.507390761802504] plot(\x,{(--3.6262606986009076--2.3868366676731565*\x)/2.587598630374626});
\draw [domain=-6.295541561786057:7.507390761802504] plot(\x,{(-21.1897240279081--1.3912804015212652*\x)/-4.044907410739923});
\draw [domain=-6.295541561786057:7.507390761802504] plot(\x,{(--9.052652261566836--1.4354160754323595*\x)/5.353689635521689});
\draw [domain=-6.295541561786057:7.507390761802504] plot(\x,{(-23.336433152348445-7.686967124283885*\x)/-2.7803230550587763});
\draw [domain=-6.295541561786057:7.507390761802504] plot(\x,{(-13.823121538270541-5.621052349141927*\x)/-3.5771950646237727});
\draw [domain=-6.295541561786057:7.507390761802504] plot(\x,{(-21.119980432508115--0.2767862186201042*\x)/-6.0141264063219895});
\draw [domain=-6.295541561786057:7.507390761802504] plot(\x,{(--3.82935328761021--2.3478598593786866*\x)/-0.06914043897792954});
\draw [domain=-6.295541561786057:7.507390761802504] plot(\x,{(-8.429716098031342--1.5803614713398195*\x)/-1.461534024303433});
\draw [domain=-6.295541561786057:7.507390761802504] plot(\x,{(--13.97744076695705--4.8878253451632405*\x)/0.8810581619647937});
\draw [domain=-6.295541561786057:7.507390761802504] plot(\x,{(--3.7685284024981756--2.601269392035054*\x)/4.649191503365426});
\draw [domain=-6.295541561786057:7.507390761802504] plot(\x,{(--22.578370086368736-2.3095271031601747*\x)/3.712783904352978});
\begin{scriptsize}
\draw [color=black] (-1.0148207931868989,5.587674083345914) circle (2.5pt);
\draw[color=black] (-0.7070659339631945,5.709150929642949) node {$A$};
\draw [color=black] (0.44248798717839866,1.8095570141514923) circle (2.5pt);
\draw[color=black] (0.10378746810102532,2.099289475761507) node {$B$};
\draw [color=black] (3.0300866175530246,4.196393681824649) circle (2.5pt);
\draw[color=black] (3.0231893977940376,4.71507360476645) node {$C$};
\draw [fill=black,shift={(5.796177622700087,3.2449730895838518)}] (0,0) ++(0 pt,3.75pt) -- ++(3.2475952641916446pt,-5.625pt)--++(-6.495190528383289pt,0 pt) -- ++(3.2475952641916446pt,5.625pt);
\draw[color=black] (5.802460034861785,3.6407336946394193) node {$E$};
\draw [fill=black,shift={(-0.21794878362190206,3.521759308203956)}] (0,0) ++(0 pt,3.75pt) -- ++(3.2475952641916446pt,-5.625pt)--++(-6.495190528383289pt,0 pt) -- ++(3.2475952641916446pt,5.625pt);
\draw[color=black] (0.10707703722593714,3.30047110938889) node {$D$};
\draw [fill=black,shift={(-3.795143848245675,-2.0992930409379706)}] (0,0) ++(0 pt,3.75pt) -- ++(3.2475952641916446pt,-5.625pt)--++(-6.495190528383289pt,0 pt) -- ++(3.2475952641916446pt,5.625pt);
\draw[color=black] (-3.9133095871565593,-1.7309658559957302) node {$F$};
\draw [fill=black] (-2.684591809853696,0.9711332512128636) circle (1.5pt);
\draw[color=black] (-2.909035323342163,1.2818569354474623) node {$G$};
\draw [fill=black] (-1.6676251409406595,1.2437994691144594) circle (1.5pt);
\draw[color=black] (-1.4010119662144292,1.1050808582597416) node {$H$};
\draw [fill=black] (2.179109096757502,3.4114402790218783) circle (1.5pt);
\draw[color=black] (2.229112072917538,3.780864987264684) node {$K$};
\draw [fill=black] (-1.736765579918589,3.591659328493146) circle (1.5pt);
\draw[color=black] (-1.4843671816519732,3.3572471865766104) node {$J$};
\draw [fill=black] (0.7175750724540692,4.991801750361698) circle (1.5pt);
\draw[color=black] (0.7110230694771718,5.369019637017685) node {$I$};
\draw [fill=black] (-1.8600762902341084,7.779024512473213) circle (1.5pt);
\draw[color=black] (-2.114957998465664,7.714738119709995) node {$L$};
\draw [fill=black] (3.0211486423482095,2.5009401927647215) circle (1.5pt);
\draw[color=black] (3.0698998286691257,2.870011585200463) node {$O$};
\draw [fill=black] (-1.8035336478889021,5.858958596376104) circle (1.5pt);
\draw[color=black] (-2.068247567590576,6.256517823644362) node {$M$};
\draw [fill=black] (-1.628042861017216,-0.10032919927033278) circle (1.5pt);
\draw[color=black] (-1.357656750776885,-0.22952163711781775) node {$N$};
\draw [fill=black] (-0.6826972867999531,6.5059207849848235) circle (1.5pt);
\draw[color=black] (-0.42009291783757707,6.7002669169577) node {$P$};
\draw [fill=black] (-3.017903127344981,-0.8779699928119649) circle (1.5pt);
\draw[color=black] (-2.819035323342163,-0.9235333156816858) node {$R$};
\draw [fill=black] (4.460785376258695,3.306431420402868) circle (1.5pt);
\draw[color=black] (4.424502324046683,3.6173784792018755) node {$S$};
\draw [fill=black] (-1.5856672542009285,7.067610790860788) circle (1.5pt);
\draw[color=black] (-1.2075254581516205,6.650332563270332) node {$Q$};
\end{scriptsize}
\end{tikzpicture}
\caption{A configuration of $12$ lines with $19$ triple points.}
\end{figure}
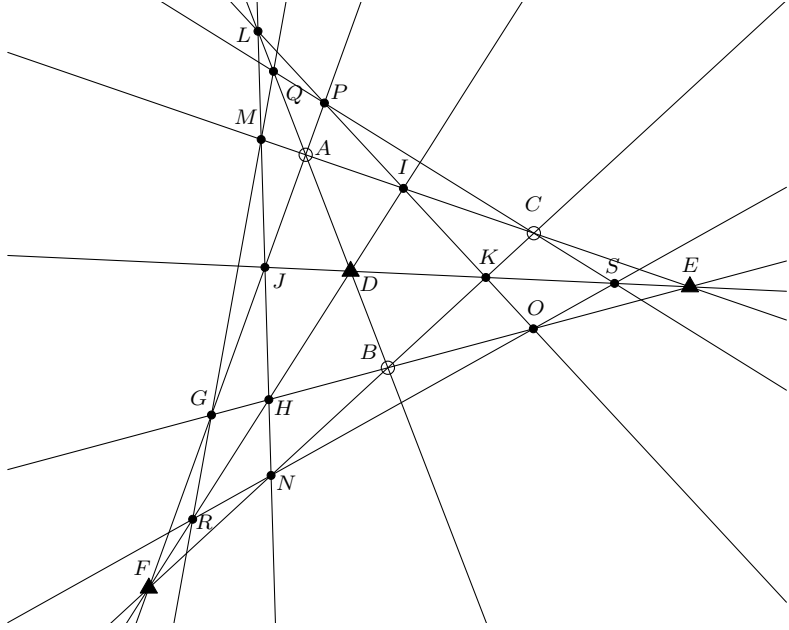

%\begin{figure}[h]
%   \fbox{\includegraphics[width=0.9\textwidth]{contfinal.jpg}}
%   \caption{Configuration of $19$ triple points and $12$ lines.}
%\end{figure}

\begin{thm}\label{realExample}
Let $\field$ be a field such that the points $A,\ldots,S\in \field[\pr2]$ specified above are distinct
and the ideal $I$ of the set $Z$ of these 19 points satisfies
$\alpha(I)=\omega(I)=5$ and $\reg(I)=7$ with $I^{(3)}\not\subseteq I^2$. Then
   $$\rho(I)=\frac{3}{2}\; \mbox{ and }\; \hatrho(I)=\frac{5}{4}.$$
\end{thm}
\begin{proof}
   We begin by computing the Waldschmidt constant $\widehat{\alpha}(I)$ of $I$
   (we will show that $\widehat{\alpha}(I)=4$).
   By way of contradiction, assume that there exists $m\geq 1$ such that
   \begin{equation}\label{eq:small alpha}
      \alpha(I^{(m)})\leq 4m-1.
   \end{equation}
   Let $D$ be a divisor of degree $d\leq 4m-1$ vanishing on $Z$ to order at least $m$.

   Since every line in the configuration contains at least $4$ configuration points,
   Bezout's Theorem implies that each configuration line is a component of $D$.
   Subtracting these $12$ lines from $D$ we obtain a divisor $D'$ of degree
   $d'=d-12$ vanishing at each point of $Z$ to order at least $m-3$. In other
   words, we are in the situation of \eqref{eq:small alpha} with $m$ replaced
   by $m'=m-3$. Indeed
   $$\alpha(I^{(m')})\leq d'=d-12\leq 4(m-3)-1=4m'-1.$$
   Continuing by a finite descent, we will be reduced to a situation
   in which $m'$ is either $1, 2$ or $3$ and the degree $d'$ is at most either
   $3, 7$ or $11$ respectively. Each of these possibilities is eliminated by one more application of
   Bezout's Theorem. Hence our assumption in \eqref{eq:small alpha}
   was false and it must be that
   $$\alpha(I^{(m)})\geq 4m$$
   for all $m\geq 1$ and hence $\widehat{\alpha}(I)\geq4$. Since the 12 lines give a form in $I^{(3)}$,
   we have $\alpha(I^{(3)})\leq12$ (and hence $\alpha(I^{(3m)})\leq m\alpha(I^{(3)})\leq12m$, so
   $\widehat{\alpha}(I)\leq4$), hence $\widehat{\alpha}(I)=4$.

   Now applying \eqref{eq:lower bound on hatrho} and \eqref{eq:upper bound on hatrho},
   we obtain
   $$\hatrho(I)=\frac54.$$

   Finally we turn to $\rho(I)$. The proof follows the same
   lines as that of Theorem \ref{thm:fermat configurations}.
   Suppose $I^{(m)}\not\subseteq I^r$. This never happens for $r=1$,
   so consider $r=2$. Since $I^{(m)}\subseteq I^2$ for $m\geq 2r$ and since we know
   $I^{(3)}\not\subseteq I^2$, we have
   $I^{(m)}\not\subseteq I^2$ if and only if $m\leq 3$ and hence $\frac{m}{r}\leq \frac32$.
   Now assume that $r>2$.
Then $\alpha(I^{(m)})<\reg(I^r)$, but we saw above that $4m\leq \alpha(I^{(m)})$
   and $\reg(I^r)\leq 2\reg(I)+(r-2)\omega(I)=5r+4$, hence
$4m< 5r+4$, or $\frac{m}{r}<\frac54+\frac1r$. If $r\geq 4$, then
$\frac{m}{r}<\frac54+\frac14=\frac32$. If $r=3$, then
$4m< 5r+4=19$, so $m\leq 4$, hence $\frac{m}{r}\leq\frac43<\frac32$.
Thus $\frac{m}{r}\leq \frac32$ in all cases, with equality in one case (namely, $m=3$, $r=2$)
so $\rho(I)=\frac32$.
\end{proof}

\section{Results in dimension $N\geq 2$ over finite fields}

Here we compute the resurgences for some ideals including a range of ideals giving
exclusively positive characteristic counterexamples
to Conjecture \ref{Bconj}.

So in this section we let $\field={\mathbb F}_s$ be a field of characteristic $p>0$ of $s$ elements 
and let $\field'={\mathbb F}_p$ be the subfield of order $p$.
Let $I\subseteq \field[\pr{N}]=\field[x_0,\ldots, x_N]$ be the ideal of
all of the $\field$-points of $\pr{N}=\pr{N}(\field)$ but one.
We recall that $I^{(Nr-(N-1))} \not\subseteq I^r$ holds for the following cases
(see \cite[Proposition 2.2 and Section 3]{refHS}):
\begin{enumerate}
\item[(i)]
$p>2$, $N=2$ and $r=(s+1)/2$;
\item[(ii)]
$s=p>2$, $r=2$ and $N=(p+1)/2$ (in which case $Nr-(N-1)=(p+3)/2$) and
\item[(iii)]
$r=(p+N-1)/N$ (in which case $Nr-(N-1)=p$), $s=p>(N-1)^2$ and $p\equiv 1 (\!\!\!\mod N)$.
\end{enumerate}

\begin{lem}\label{reglemma}
Let $I$ be the ideal of all but one of the $\field$-points of ${\bf P}^N(\field)$;
let $q$ be the excluded point. Then  $\reg(I) = N(s-1)+1$.
\end{lem}

\begin{proof}
%Let $C'$ be the reduced scheme of all
%$\field$-points of ${\bf P}^n(\field)$ other than $q$ (so $I_{C'}=I$),
Let $J$ be the defining ideal of the set of $\field$-lines through $q$,  and let $H$ be a hyperplane not passing through $q$. Without loss of generality we may assume that $q=[1:0:\ldots:0]$ and that $H$ is defined by $x_0$. %Let $B$ be the defining ideal of the set of $\field$-points (other than $q$) not on $H$.
The defining ideal of the set of points of ${\bf P}^N(\field)$ off $H$ and excluding
$q$ is given by $B=C:I_q$, where $I_q$ is the defining ideal of the point $q$ and
$C$ is the ideal defining the set of points ${\bf P}^N(\field)\setminus \{H\}$,
namely the complete intersection
$C=(x_1(x_1^{s-1}-x_0^{s-1}),\ldots,x_N(x_N^{s-1}-x_0^{s-1}))$. To see that $C$ is
precisely the ideal indicated before, note that both are unmixed ideals of the same
degree which satisfy an obvious containment. The relation $B=C:I_q$ yields that
$B$ is linked to $I_q$ via the complete intersection $C$. As a consequence of a
well-known formula for the behavior of Hilbert functions under linkage, we have,
as in \cite[Theorem 3]{refDGO}, that $\alpha(I_q/C) + \reg (R/B) =\reg (R/C)$.  The Koszul resolution shows that $\reg (R/C) = N (s-1)$. Since $\alpha(I_q/C)=1$, we conclude  $\reg (B) = 1 + \reg (R/B) =  N (s-1)$.

By \cite[Lemma 4.7]{refHS}, $I$ is a basic double link of $J$, i.e., $I=x_0 B+J$. It follows that 
there is a short exact sequence 
\[
0 \to (R/B) (-1) \to R/I \to R/(x_0, J) \to 0, 
\]
where the embedding is induced by multiplication by $x_0$. Taking cohomology we get 
\begin{equation}
\label{eq:reg all but one}
  \reg (I) = \max\{1 + \reg (B), \reg (x_0, J)\} = \max\{1 + \reg (B), \reg (J)\}.
\end{equation}

In order to compute the regularity of $J$ we use induction on $N$. Let $D$ be the ideal of  all $\field$-points of ${\bf P}^N(\field)$. We claim $\reg (D) = N (s-1) + 2$. 

Indeed, the ideal $x_0 C + J$ is a basic double link of $C$. Thus, it is saturated of degree 
\[
\deg C + \deg J = \deg D. 
\]
Since $x_0 C + J \subset D$ and both saturated ideals have the same degree, we get $x_0 C + J = D$. As above, this gives 
\begin{equation}
\label{eq:reg all points}
   \reg (D) = \max \{ 1 + \reg (C), \reg (J)\}.
\end{equation}

Now observe that $J$ is the defining ideal of the cone in ${\bf P}^N(\field)$ over the $\field$-points in the hyperplane $H$. Hence, the induction hypothesis yields $
\reg (J) = (N-1)(s-1) +2$. Using, $\reg (C) = N (s-1) + 1$, Equation \eqref{eq:reg all points} provides $\reg (D) = N (s-1) + 2$, as claimed.

Finally, applying Equation \eqref{eq:reg all but one}, we obtain 
\[
\reg (I) = \max\{N (s-1) +1, (N-1)(s-1) +2\} = N (s-1) +1,
\]
as desired.
 \end{proof}

\begin{thm}\label{mainThmSec3}
Let $I$ be the ideal of all but one of the $\field$-points of ${\bf P}^N(\field)$.
Then   $\rho(I) = \widehat{\rho}(I)= \frac{N(s-1)+1}{s}$ and $\ha{I}=s$.
\end{thm}

\begin{proof}
Let  $q$ be the excluded point and let $F$ be the product of all hyperplanes
defined over $\field$ but not vanishing at $q$, so $\deg(F) = s^N$. Since $F$
vanishes with multiplicity $s^{N-1}$ at each non-$q$ point, we have $F^{(N(s-1)+1)t}\in I^{((N(s-1)+1)ts^{N-1})}$.
By the argument of \cite[Proposition 3.8]{refHS} (which assumes $s=p$ but works also for $s>p$),
$I$ vanishes at all $\field$-points in degrees less
than $N(s-1)+1$, hence $I^{s^Nt+1}$ vanishes at $q$ in degrees less than
$(s^Nt+1)(N(s-1)+1)$. Since $\deg(F^{(N(s-1)+1)t}) = (N(s-1)+1)ts^N < (s^Nt+1)(N(s-1)+1)$,
we obtain $I^{((N(s-1)+1)s^{N-1}t)}\not\subseteq I^{s^Nt+1}$, thus
$\frac{(N(s-1)+1)s^{N-1}t}{s^Nt+1} \leq \rho(I)$ for all $t$, hence, after passing
to the limit as $t\to\infty$, we obtain $\frac{N(s-1)+1}{s}  \leq \rho(I)$.
%Since $F\in I^{(s^{n-1})}$ and $\deg(F) = s^n$, we also have $\ha{I}\leq \frac{s^n}{s^{n-1}}=s$.

%Note: the original argument in \cite[Proposition 4.8]{refHS} is for |k|=p
%but  the argument works for any finite field k so we can replace p by s.

To show that  $ \rho(I)\leq N - \frac{N-1}{s}$, it suffices to prove that $I^{(m)}\subseteq I^r$,
whenever $\frac{m}{r} > \frac{N(s-1)+1}{s}$, that is whenever $ms>r(N(s-1)+1)$.
Recall that by Lemma \ref{reglemma} we have $\reg(I)=N(s-1)+1$ and, as a consequence of
work of \cite{refCh,refGGP} improved upon in \cite[Proposition 1.7.1]{refC1}, it follows
that $\reg(I^r)\leq r \reg(I)=r(N(s-1)+1)$ for any positive integer $r$. (Note that the
preceding inequality is guaranteed to hold only for homogeneous ideals $I$ with $\dim(R/I)\leq 1$,
a hypothesis which is satisfied by our ideal.) Without loss of generality we may assume that
$q=[1:0:\ldots:0]$. Next, note that the ideal $I$ is contained in the complete intersection
$C=(x_1(x_1^{s-1}-x_0^{s-1}),\ldots,x_N(x_N^{s-1}-x_0^{s-1}))$ defining the $s^N$
points of ${\bf P}^N(\field)$ that are not situated on $H=V(x_0)$ and are distinct from
$q$. Thus $I^{(m)}\subseteq C^{(m)}=C^m$ and so $\alpha(I^{(m)})\geq \alpha(C^m)=m\alpha(C)=ms$.
Combining the three inequalities gives
$\alpha(I^{(m)})\geq ms >r(N(s-1)+1) \geq \reg(I^r)$. By \cite[Lemma 2.3.4]{refBH},
$\alpha(I^{(m)})>  \reg(I^r)$ implies $I^{(m)}\subseteq I^r$ as desired.

Now we show that $\widehat{\rho}(I)=\rho(I)=N - \frac{N-1}{s}$.
We know that $\widehat{\rho}(I)\leq \rho(I)=N - \frac{N-1}{s}$. It remains to see that the opposite inequality holds.
Recall from the first paragraph of this proof  that
$I^{((N(s-1)+1)s^{N-1}t)}\not\subseteq I^{s^Nt+1}$ for all $t>0$.
Now  for $u,v>0$, letting  $t=uv$, we deduce that $I^{((N(s-1)+1)s^{N-1}uv)}\not\subseteq I^{s^Nuv+1}$.
As a consequence,  $I^{((N(s-1)+1)s^{N-1}uv)}\not\subseteq I^{s^Nuv+u}=I^{(s^Nv+1)u}$,
because $I^{(s^Nv+1)u}\subseteq  I^{s^Nuv+1}$. Thus we have
$\frac{(N(s-1)+1)s^{N-1}v}{(s^Nv+1)}\leq \widehat{\rho}(I)$
for all $v>0$ and hence $\frac{N(s-1)+1}{s}=\lim_{v\to\infty}\frac{(N(s-1)+1)s^{N-1}v}{s^Nv+1}\leq\widehat{\rho}(I)$.

To finish, note by the argument in the first paragraph above that
$F^t\in I^{(ts^{N-1})}$, hence $\frac{\alpha(I^{(ts^{N-1})})}{t}\leq \frac{\deg(F^t)}{t}=s$,
so taking the limit as $t\to\infty$ gives $\ha{I}\leq s$. But we also saw
that $\alpha(I^{(m)})\geq \alpha(C^m)=m\alpha(C)=ms$,
so $\frac{\alpha(I^{(m)})}{m}\geq s$, hence also $\ha{I}\geq s$.
\end{proof}

\end{document}